\documentclass{amsart}

\usepackage{amsmath, amssymb, graphicx, epsfig, verbatim}
\usepackage{epic}
\usepackage{color}

\newtheorem{thm}{Theorem}[section]

\newtheorem{cor}[thm]{Corollary}
\newtheorem{lem}[thm]{Lemma}
\newtheorem{prop}[thm]{Proposition}

\newtheorem{exa}[thm]{Example}
\newtheorem{rem}[thm]{Remark}

\numberwithin{equation}{section}

\newcommand{\Z}{\mathbb Z}
\newcommand{\N}{\mathbb N}

\newcommand{\cpkk}{{\overline {{\mathbb C}{\mathbb P}^2}}}
\newcommand{\cpk}{{\mathbb {CP}}^2}
\newcommand{\rpk}{{\mathbb {RP}}^2}
\newcommand{\cpegy}{{\mathbb {CP}}^1}

\begin{document}

\title{Exotic definite four-manifolds with non-trivial fundamental group}

\author{Andr\'{a}s I. Stipsicz}
\address{R\'enyi Institute of Mathematics\\
H-1053 Budapest\\ 
Re\'altanoda utca 13--15, Hungary}
\email{stipsicz.andras@renyi.hu}

\author{Zolt\'an Szab\'o}
\address{Department of Mathematics\\
Princeton University\\
 Princeton, NJ, 08544}
\email{szabo@math.princeton.edu}

\begin{abstract}
  Inspired by a recent result of Levine-Lidman-Piccirillo~\cite{LLP},
  we construct infinitely many exotic smooth structures on some closed
  four-manifolds with definite intersection form and fundamental group
  isomorphic to $\Z /2\Z$. Similar constructions provide exotic smooth
  structures on further four-manifolds with fundamental group $\Z/2\Z$,
  including examples with even $b_2^+$.
\end{abstract}
\maketitle

\section{Introduction}
\label{sec:intro}

The closed, smooth four-manifolds $X$ and $Y$ form an exotic pair (or
$Y$ is \emph{an exotic} $X$, or $X$ admits an exotic smooth structure
$Y$) if $X$ and $Y$ are homeomorphic, but not diffeomorphic. Many such
exotic pairs (or even collections of infinitely many homeomorphic but
pairwise non-diffeomorphic four-manifolds) have been constructed in
the past decades; the homeomorphism is typically verified using (a
variant of) Freedman's groundbreaking result \cite{Fr}, while
non-diffeomorphism is shown by the computation of a smooth
four-manifold invariant, like Donaldson's polynomial invariants
\cite{Do}, Seiberg-Witten invariants \cite{W}, or the mixed invariant
from Heegaard Floer theory \cite{OSzFour}.

The constructions usually rely on some cut-and-paste operation, and
for many simply connected topological manifolds infinitely many
distinct smooth structures have been found in this way.
The existence question of exotic structures seems particularly hard
for manifolds with definite intersection form (including the four-sphere
$S^4$, the complex projective plane $\cpk$, and the connected sums
$\# n \cpk$).

Recently, by stacking elementary cobordisms on top of each other,
Levine-Lidman-Piccirillo~\cite{LLP} found an example of an exotic
oriented definite four-manifold with fundamental group isomorphic
to $\Z/2\Z$.  The first result of this note is an
alternative construction of infinitely many exotic definite
manifolds.  In order to state the theorem, let us consider the
quotient $Z_0$ of $S^2\times S^2$ with the free $\Z/2\Z$-action given by
the action of the antipodal map on one coordinate and by conjugation
on the other.

\begin{thm}\label{thm:main}
  The smooth, definite four-manifold $Z_0\# p\cpkk$ (with $\pi
  _1=\Z/2\Z$) admits infinitely many distinct smooth structures once
  $p\geq 4$. The exotic smooth structures on $Z_0\# 4\cpkk$ are
  irreducible.
 \end{thm}
\begin{rem}
  The construction presented here was motivated by results and ideas
  in \cite{LLP}. In that paper one example of such an exotic structure
  was given for $p=4$. To our knowledge, the above theorem provides the
  first example of infinitely many exotic structures on a closed, oriented
  four-manifold with definite intersection form.
  It would be particularly interesting to find
  exotic examples on $Z_0\# p\cpkk$ with $p<4$. 
  \end{rem}

Indeed, the above result is a special case of a more general theorem:
\begin{thm}\label{thm:general}
  The four-manifolds $Z_0\# (n-1)\cpk \# p\cpkk$ with $p\geq 5n-1$
  ($n\in \N$) and the spin four-manifolds $Z_0\# qK3\# (q-1)S^2\times
  S^2$ ($q\in \N$) (with fundamental groups isomorphic to $\Z/2\Z$) 
  carry infinitely many distinct smooth structures.
\end{thm}

\begin{rem}
  Exotic four-manifolds with non-trivial fundamental groups have been
  extensively studied in \cite{RT1, RT2}. Some of the exotic phenomena
  encountered in Theorem~\ref{thm:general} have been already encountered
  in \cite{RT1, RT2}; the examples with $b_2^+$ even above are, however,
  new.
\end{rem}

Similar, but slightly different  methods
(resting on variations of the rational blow-down construction)
provide further interesting exotic examples
--- as a sample
of such results, we have the following

\begin{thm}\label{thm:MoreEx}
  For a given $n\in \N$ there are infinitely many non-diffeomorphic
  four-manifolds all homeomorphic to $Z_0\#2n\cpk \# 8n\cpkk$.
\end{thm}

Our proof of Theorem~\ref{thm:main} will rely on the knot surgery
construction of Fintushel and Stern \cite{FSknot}.
By the rational blow-down construction we can
construct further exotic structures on $Z_0\# 4\cpkk$ which cannot be
given by the knot surgery method.

\begin{thm}\label{thm:theDns}
  The smooth manifolds $D_n$ for $n$ even constructed in
  Section~\ref{sec:RatBlowDown} via the rational blow-down
  construction are all homeomorphic to $Z_0\# 4\cpkk$, but not
  diffeomorphic to the four-manifolds constructed in the proof of
  Theorem~\ref{thm:main} in Section~\ref{sec:construction}.
  \end{thm}

In proving homeomorphisms of four-manifolds with fundamental
  group $\Z/2\Z$ we will use the result of Hambleton and Kreck~\cite{HK}
  recalled below. For showing non-diffeomorphism, we will apply Seiberg-Witten
  invariants for the universal (double) covers. Indeed, our constructions
  will give simply connected four-manifolds equipped with
  free, orientation preserving involutions, and the examples
  claimed by the theorems are the quotient spaces by these $\Z/2\Z$-actions.
  The simply connected manifolds, in turn, are constructed by
  knot surgeries, rational blow-downs, and double node surgeries
  \cite{FSRatBl, FSknot, FSDoubleNode} in a $\Z/2\Z$-equivariant
  manner.

We quickly recall the homeomorphism classification of oriented smooth
four-manifolds with fundamental group $\Z /2\Z$.  In
the simply connected case there are two types of four-manifolds: odd
(Type I) and even (Type II), i.e. non-spin and spin.  The
\emph{$w_2$-type} of a closed, oriented four-manifold $X$ with $\pi
_1(X)=\Z/2\Z$ (or actually with finite fundamental group) is Type I if
the universal cover ${\widetilde {X}}$ is not spin ($w_2({\widetilde
  {X}})\neq 0$). It is Type II if $X$ is spin (and so ${\widetilde
  {X}}$ is also spin), and finally it is Type III if $X$ is non-spin
($w_2(X)\neq 0$) but ${\widetilde {X}}$ is spin $(w_2({\widetilde
  {X}})=0$). With this terminology in place, the topological
classification of smooth, closed, oriented four-manifolds with
fundamental group $\Z/2\Z$ given below easily follows from \cite[Theorem~C]{HK},
together with Donaldson's Diagonalizability Theorem~\cite{Dona}
(and the classification of indefinite unimodular forms, cf. \cite{GS}):

\begin{thm}(\cite{HK})
  \label{thm:classificationHK}
  Suppose that $X_1, X_2$ are two smooth, closed, oriented four-manifolds
  with fundamental group $\pi _1(X_1)=\pi _1(X_2)=\Z/2\Z$. Then $X_1$ and
  $X_2$ are homeomorphic if and only if they have the same $w_2$-type,
  the same signature and the same Euler characteristic.
\end{thm}

The paper is organized as follows: in Section~\ref{sec:construction}
we describe the construction of our exotic manifolds and prove
Theorems~\ref{thm:main} and \ref{thm:general}. In
Section~\ref{sec:another} we define another involution on certain
elliptic surfaces and in Sections~\ref{sec:RatBlowDown} (when combined
with the rational blow down construction and/or double node surgeries)
we find further exotica using this approach, leading to the proofs of
Theorems~\ref{thm:MoreEx} and \ref{thm:theDns}.

\bigskip

\noindent {\bf {Acknowledgements}:} The first author was partially
supported by the \emph{\'Elvonal (Frontier) grant}  KKP144148 of the
NKFIH.  The second author was partially supported by NSF grant
DMS-1904628 and the Simons Grant \emph{New structures in
  low-dimensional topology}. We would like to thank Mark Powell for
useful correspondence, and the organizers of the conference ''\emph{Gauge
Theory and Topology: in Celebration of Peter Kronheimer’s 60$^{th}$
Birthday}'' in Oxford, where this collaboration started.

\section{The construction}
\label{sec:construction}

The main ingredient of our construction is the identification of a
smooth, fixed point free involution on the elliptic surface $E(n)$.
(In the following $E(n)$ denotes the simply connected elliptic surface
with a section, Euler characteristic $12n$ and signature $-8n$.)
We will get Theorem~\ref{thm:main} by specializing this action to the 
rational elliptic surface $E(1)=\cpk \# 9\cpkk$.

For the definition, we present the elliptic surface $E(n)$ as the
double branched cover of $S^2\times S^2$, branched along the following
curve: take four points $a_1, a_2, a_3, a_4\in S^2$ and consider
$H=\cup _{i=1}^4\{ a_i\} \times S^2\subset S^2\times S^2$.  Take
further $2n$ points $b_1, \ldots , b_{2n}$ and consider $V=\cup
_{j=1}^{2n} S^2\times \{ b_j\}$. Define the branch curve now to be $H\cup
V\subset S^2\times S^2$. As this is a singular curve (with $8n$
transverse double points), the double branched cover will have $8n$
simple $A_1$ singularities. We get a smooth manifold either by
resolving these singular point (i.e.  topologically replacing the
cone neighborhood of each singular point
with the cotangent disk bundle of $S^2$), or by blowing up the
transverse double points in the branch locus $H\cup V$ and taking
the double branched cover along the strict transforms.  We will follow this
second route, and consider the proper transform of $H\cup V$ in
$S^2\times S^2\#8n\cpkk$. By composing the branch cover map with the
projection to the second factor, we get an elliptic fibration on the
branched cover. This fibration has $2n$ singular fibers, each of type
$I_0^*$ (in the terminology of \cite{HKK}); topologically such a
singular fiber is a plumbing along the tree with five vertices, four
of which are leaves and the fifth of degree 4, and all framings are
$-2$.

Consider now the antipodal map $p\colon S^2\to S^2$ and the complex
conjugation $c\colon \cpegy \to \cpegy $. The former map is free,
while $c$ has the real circle in $\cpegy=S^2$ as fixed point set.
Define $j\colon S^2\times S^2\to S^2\times S^2$ as $j=(p,c)$.  It is a
free action on $S^2\times S^2$ (studied in \cite{Lev}, for example),
with $Z_0=S^2\times S^2/j$. The manifold $Z_0$ is indeed a spin
manifold and admits an $S^2$-fibration over $\rpk$; a fixed point $x$
of $c$ defines a sphere $S^2\times \{ x\}$ which gives rise to a
section of this fibration with self-intersection 0. (More on
$S^2$-fibrations over $\rpk$ see \cite[page~237]{Hil}; a Kirby diagram
of $Z_0$ will be given in Figure~\ref{fig:kirby}.)

Assume that the points $\{ a_1, \ldots , a_4\}$ and $\{ b_1, \ldots ,
b_{2n}\}$ are chosen so that these sets are invariant under $p$ and
$c$ respectively (and none of them are fixed points of $c$), that is,
each $a_i$ is mapped to some other $a_j$ under $p$ and similarly each
$b_i$ is mapped to some other $b_j$ by $c$.  We claim that the
involution $j$ lifts to an involution on $E(n)$.  For this, we need a
lemma.

\begin{lem}\label{lem:lift}
  Suppose that $b\colon X\to Y$ is a double branched cover of
  connected manifolds along the branch set $B\subset Y$. Suppose
  furthermore that $Y$ admits a smooth involution $\tau\colon Y\to Y$
  and $B$ is invariant under $\tau$ as a set. Assume that the double
  branched cover $X\to Y$ is given by the representation $\phi\colon
  H_1(Y\setminus B; \Z )\to \Z/2\Z$, and the induced map $\tau _*$ on
  $H_1(Y\setminus B; \Z )$ commutes with $\phi$. Then $\tau$ lifts to
  a diffeomorphism $\sigma \colon X\to X$ with either $\sigma ^4={\rm
    {Id}}_X$ or $\sigma ^2={\rm {Id}}_X$.
\end{lem}
\begin{proof}
  Consider a point $x\in X\setminus b^{-1}(B)$ with $y=b(x)\in Y\setminus B$.
  Take $\tau (y)$ and define $\sigma (x)$ to be one of the preimages
  of $\tau (y)$ under $b$. For a point $x'\in b^{-1}(B)$, define
  $\sigma (x')=b^{-1}(\tau (b(x')))$. For any other point $x'\in
  X\setminus b^{-1}(B)$ consider a path $\gamma$ from $x$ to $x'$ in
  $X\setminus b^{-1}(B)$, and consider $b (\gamma )$ connecting $y$
  and $y'= b(x')$. Apply $\tau $ to $b(\gamma)$, giving a path in
  $Y\setminus B$ from $\tau (y)$ to $\tau (y')$. Lift this path to
  $X\setminus b^{-1}(B)$ starting at $\sigma (x)$; the endpoint of this lifted
  path will be $\sigma (x')\in X$.  This point will be independent of
  the chosen path $\gamma$; indeed, if $\gamma '$ is another path from
  $x$ to $x'$, then the loop $\gamma ^{-1}*\gamma '$ projects to a
  loop in $Y\setminus B$ which lifts to a loop, hence its $\tau$-image
  in $Y\setminus B$ also lifts to a loop (as $\tau _*$ commutes with
  $\phi$).

  As the lift $\sigma$ interchanges the preimages of points of
  $Y\setminus B$, which consist of two points, $\sigma$ is either of
  order 2 (i.e. an involution) or of order 4.
\end{proof}

\begin{exa}\label{exa:square}
Consider the double branched cover map $b\colon T^2\to S^2$ between
the 2-torus and the 2-sphere, branched in four points. Consider the antipodal
involution  $p\colon S^2\to S^2$ and assume that it permutes the branch points.
Consider a main circle $C$ separating the four branch points into two groups
of two; $C$ is obviously fixed setwise by $p$. Its inverse image in
$T^2$ consists of two disjoint circles, and the inverse image of a point
$x\in C$ lifts to one point in each component. Now the lift $\sigma$
of $p$ either fixes the two components (setwise) or interchanges them,
hence $\sigma ^2$ fixes them, and hence fixes the lift of $x$, implying that
$\sigma ^2=$Id, so the lift $\sigma$ is an involution.
\end{exa}

\begin{exa}
  Straightforward modification of the above argument shows that if $b\colon
  T^2\to S^2$ is the above branched cover, but the involution on $S^2$
  is the conjugation map $c$, then its lift $\sigma$ is of order
  four. In a similar manner, if we consider $b\colon S^2\to S^2$
  branched in two points, with the antipodal map as involution
  downstairs, then the lift is again of order four.
\end{exa}  
  
Consider now our earlier situation of $S^2\times S^2$ with the curve
$H\cup V$.  Note that as $j$ is an antiholomorphic map, it maps
directions through $x$ to directions through $j(x)$, hence naturally
extends to the simultaneous blow-up of $x$ and $j(x)$.  Extend the
involution $j$ from $S^2\times S^2$ to the blown up manifold
$S^2\times S^2\# 8n\cpkk$ (where we blow up all the transverse double
points of $H\cup V$) in the obvious way; denote the resulting
involution still with $j$.  By the choices of the points $a_i$ and
$b_k$, the branch locus is $j$-invariant, and the exceptional divisors
of the blow-ups map to each other (inducing an orientation reversing
map on the exceptional curves).  As the first homology group of the
complement of the branch locus is isomorphic to $\Z$,
Lemma~\ref{lem:lift} lifts $j$ to a free action on the double branched
cover. By taking a sphere $S^2\times \{ b \}$ for $b\in S^2$ with
$c(b)=b$, Example~\ref{exa:square} shows that we get a free,
orientation preserving involution on $E(n)$, which we will still
denote by $j$.  This self-diffeomorphism preserves fibers, but
reverses the orientation on them.  If $n$ is not divisible by 4, then
the quotient will be homeomorphic to $Z_0\# (n-1)\cpk \# (5n-1)\cpkk$,
with fundamental group $\Z /2\Z$, Euler characteristic $6n$ and
signature $-4n$. For odd $n$ this is a non-spin manifold with non-spin
universal cover (Type I), while if $n$ is even and not divisible by
four, it is non-spin with spin universal cover (Type III).  If $n=4q$
then the quotient is homeomorphic to $Z_0\#q K3\# (q-1)S^2\times S^2$;
a spin manifold with spin universal cover (i.e. of Type II).  Note
that the quotient $E(n)/j$ is an exotic smooth structure on these
manifolds once $n>1$, since the universal (double) covers of the above
manifolds decompose as connected sums, while $E(n)$ is minimal for all
$n>1$ (shown by the Seiberg-Witten basic classes of these manifolds).

Further (infinite families of) exotic structures can be constructed on
these manifolds as follows.  Consider two regular fibers in the
elliptic fibration on $E(n)$ resulting from the above double branched
cover construction, so that they map to each other under the
involution $j$.  For $m>0$ let $K_m$ denote the twist knot with
symmetrized Alexander polynomial $\Delta _{K_m}(t)=mt-(2m-1)+mt^{-1}$,
having $2m+1$ crossings in its alternating projection.  (Indeed, any
collection of infinitely many knots with pairwise sufficiently
distinct Alexander polynomial would be an equally good choice.)
Applying two knot surgeries (as described in \cite{FSknot}) along the
chosen fibers of $E(n)$ with the same knot $K_m$, we get a sequence of
four-manifolds $\{ X_m(n)\}_{m\in \N}$ which are all homeomorphic to
$E(n)$.

\begin{thm}\label{thm:distinctXm}
  For a fixed $n\in \N$ the smooth four-manifolds $X_m(n)$ with $m>0$
  are all homeomorphic to $E(n)$ and are pairwise non-diffeomorphic.
\end{thm}
\begin{proof}
  It is a simple fact that knot surgery does not change type, Euler
  characteristic and signature. As the fibers in the above
  construction can be chosen in cusp neighbourhoods, the manifolds
  $X_m(n)$ are all simply connected. Then Freedman's classification of
  simply connected smooth four-manifolds up to homeomorphism \cite{Fr}
  implies that the two surgeries will not change the topologcial type.

  The calculation of the Seiberg-Witten invariants of $X_m(n)$ is
  described in \cite[Theorem~1.1]{FSknot} for $n>1$ and in
  \cite[Section~5]{FSknot} when $n=1$. Indeed, this latter calculation
  shows that the Seiberg-Witten function $SW_{X_m(1)}$ takes the value
  $\pm m^2$ on the spin$^c$ structure Poincar\'e dual to $\pm 3$-times
  the fiber $F$, and $SW_{X_m(1)}$ takes the value $\pm 2m(2m-1)$ on $PD(\pm F)$.

  These results show that for $m>0$ the four-manifold $X_m(n)$ is an
  exotic $E(n)$, and if $m\neq m'$ then $X_m(n)$ and $X_{m'}(n)$ are
  smoothly distinct.
\end{proof}

The choice of the fibers in the knot surgery operation (and the fact
that we use the same knots for both fibers) ensures that the
involution $j$ extends from the complement of the pair of fibers to
$X_m(n)$ for all $m, n\in \N$ as an orinetation preserving, free
involution.  Let $W_m(n)$ denote the smooth oriented manifold we get
as the quotient of $X_m(n)$ by this extension.

\begin{thm}\label{thm:distinctWm}
  For a fixed $n$ the smooth four-manifolds $\{ W_m(n) \} _{m\in \N}$
  are homeomorphic to $Z_0\#(n-1)\cpk\# (5n-1)\cpkk$ if $n$ is not
  divisible by 4, and to $Z_0\#qK3\# (q-1)S^2\times S^2$ if $n=4q$.
  For a fixed $n$ the four-manifolds $\{ W_m(n)\}_{m\in \N}$ are
  pairwise non-diffeomorphic and are all irreducible. 
\end{thm}
\begin{proof}
  The topological classification of closed smooth four-manifolds with
  fundamental group $\Z/2\Z$ \cite[Theorem~C]{HK}, as given in
  Theorem~\ref{thm:classificationHK}, implies that for fixed $n\in \N$
  the manifolds $W_m(n) $ are all homeomorphic to the manifolds given
  in the statement.  As by construction the universal cover of
  $W_m(n)$ is $X_m(n)$, Theorem~\ref{thm:distinctXm} implies that for
  a fixed $n$ the $W_m(n)$'s ($m\in \N$) are pairwise
  non-diffeomorphic.  Irreducibility of $W_m (n)$ follows from the
  irreducibility of $X_m(n)$, which in turn is a simple consequence of
  its Seiberg-Witten invariants.
  \end{proof}

\begin{proof}[Proof of Theorem~\ref{thm:main}]
  For a fixed $p\geq 4$ consider the smooth four-manifolds
\[
\{ W_m(1)\# (p-4)\cpkk\mid m\in \N \} .
\]
    The homeomorphism of $W_m(1)\# (p-4)\cpkk$ with $Z_0\# p\cpkk$
    follows from Theorem~\ref{thm:distinctWm}. The universal cover of
    $W_m(1)\# (p-4)\cpkk$ is diffeomorphic to $X_m(1)\# 2(p-4)\cpkk$.
    A partial understanding of the Seiberg-Witten invariant of
    $X_m(1)\# k\cpkk$ will be sufficient to conclude the
    argument. Indeed, if the Poincar\'e dual of the first Chern class
    of a spin$^c$ structure on $X_m(1)\# k\cpkk$ is $F\pm e_1\pm
    \ldots \pm e_k$, then the absolute values of the Seiberg-Witten
    values on this spin$^c$ structure and for any chamber are from the
    set $\{ 2m(2m-1), 2m(2m-1)\pm 1\}$. On the other hand, if $m'<m$,
    then for any spin$^c$ structure and any chamber for $X_{m'}(1)\#
    k\cpkk$ the absolute value of the Seiberg-Witten invariant is
    $<2m(2m-1)-1$, distinguishing $X_m(1)\# k\cpkk$ and $X_{m'}(1)\#
    k\cpkk$. This observation concludes the proof.
  \end{proof}

\begin{proof}[Proof of Theorem~\ref{thm:general}]
The statements of the theorem are verified in the proof of
Theorem~\ref{thm:distinctWm} with the exception of exotic structures on
$Z_0\#(n-1)\cpk \# p\cpkk $ with $p>5n-1$. These manifolds can be constructed
by repeated blow-ups, and exoticness follows from the blow-up formula
for the universal covers.
\end{proof}

\section{Another involution on elliptic surfaces}
\label{sec:another}
There is another involution we can use to achieve similar results; in
this picture we adapt other techniques (such as the rational
blow-down method) to find interesting exotica.  For this other
involution take the antipodal maps on the two components of $S^2\times
S^2$ giving a fixed point free involution $\iota$ on this
four-manifold.  Let $Z_1$ denote the quotient $S^2\times S^2/\iota$.
After connect summing $S^2\times S^2$ with two copies of $4\cpkk$ in
two points in the same $\Z /2\Z$-orbit, we get $\cpk \# 9\cpkk$ with
an orientation preserving free involution, also denoted by $\iota$,
with $Z_1\# 4\cpkk$ as quotient.

\begin{rem}
  We could also use the same map $j$ on $S^2\times S^2$ as before; we
  chose $\iota$ because of its symmetry. Note that $Z_1$ above differs
  from $Z_0=S^2\times S^2/j$ we considered earlier as $Z_0$ is spin,
  while the diagonal $\{ (x,x)\mid x\in S^2\}\subset S^2\times S^2$
  gives rise to a copy of $\rpk \subset Z_1$ with self-intersection 1
  (in homology with $\Z/2\Z$-coefficients), consequently $Z_1$ is
  non-spin. Indeed, both $Z_0$ and $Z_1$ admit an $S^2$-fibration
  over $\rpk$, and their Kirby diagrams are given in
  Figure~\ref{fig:kirby} (cf. \cite[Figure~5.46]{GS}). The diagram
  also shows that $Z_0\# \cpkk$ is diffeomorphic to $Z_1\# \cpkk$.
  \end{rem}

\begin{figure}[htb]
\begin{center}
\setlength{\unitlength}{1mm}
\includegraphics[height=4cm]{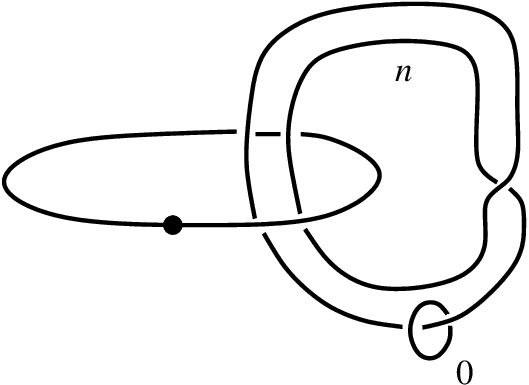}
\end{center}
\caption{\quad A Kirby diagram for $Z_0$ (when $n$ is even) and for
  $Z_1$ (when $n$ is odd). By blowing up the crossing of the two
  framed circles, a simple argument shows that $Z_0\# \cpkk$ and
  $Z_1\# \cpkk$ are diffeomorphic.}
\label{fig:kirby}
\end{figure}

Return now to the involution $\iota =(p, p)\colon S^2\times S^2 \to
S^2\times S^2$ described above, where $p\colon S^2\to S^2$ is the
antipodal map.  Fix two pairs of point $a, b\in S^2$ and $c,d\in S^2$
so that $b\neq p (a)$ and $c\neq p(d)$, that is, $a$ and $b$ (and
similarly $c$ and $d$) are \emph{not} antipodal.  (The special case
$a=c$ is also allowed.)  Take the curves
\[
C_0=(\{a\}\times S^2)\cup (\{ b \} \times S^2)\cup (S^2\times \{c\}) \cup
(S^2\times \{ d \})
\]
and
\[
C_1= (\{p (a)\}\times S^2) \cup (\{ p (b) \} \times
S^2)\cup (S^2\times \{p (c) \}) \cup
(S^2\times \{ p (d) \}),
\]
and take the pencil on $S^2\times S^2$ generated by $C_0$ and $C_1$.

The two curves intersect in 8 points, which form four pairs under the
action of $\iota$. Consider the elliptic pencil on $S^2\times S^2$
generated by $C_0$ and $C_1$; this pencil will have the 8 points
$C_0\cap C_1$ as base points.  By blowing up a pair (meaning two
blow-ups, i.e. a connected sum with $2\cpkk$) of these base points,
the involution naturally extends; repeating it three more times we get an
elliptic fibration 
on $S^2\times S^2\# 8\cpkk = \cpk \#
9\cpkk =E(1)$,
together with an orientation preserving free involution
which maps the strict transform of $C_0$ to $C_1$ and vice
versa.
Indeed, the free involution (still denoted by $\iota$)
will map the fiber over $[t_1:t_2]\in {\mathbb {CP}}^1$ to the fiber
over $[t_2:t_1]\in {\mathbb {CP}}^1$. (Note that the fibers over the
two points $[1:\pm 1]$ will be fixed setwise but not pointwise.)  The
fibration will have two $I_4$ fibers (for terminology regarding
elliptic fibrations, see \cite{HKK}), which are the
strict transforms of $C_0$ and $C_1$. By the classification of
potential combinations of singular fibers on $E(1)$ 
(together with the symmetry provided by $\iota$) from \cite{SSS}
it follows that the
resulting fibration has four fishtail (i.e. $I_0$) fibers next to the
two $I_4$'s.

\begin{rem}
  The pencil generated by $C_0, C_1$ on $S^2\times S^2$ can be given by
  a pencil on $\cpk$ as follows: take two distinct points $Q_1, Q_2\in
  \cpk$ and consider three generic lines $L_1, L_2, R_3$ through
  $Q_1$ and three generic lines $R_1, R_2, L_3$ through $Q_2$. Then the
  pencil is given by $P_0=L_1\cup L_2\cup L_3$ and $P_1=R_1\cup R_2\cup R_3$.
  Indeed, blowing up $Q_1, Q_2\in \cpk$ we get a pencil on $\cpk \# 2\cpkk=
  S^2\times S^2\# \cpkk$, which is the same as the pencil given
  by $C_0$ and $C_1$ after blowing up one of their intersection points.
  We will work with the pencil on $S^2\times S^2$ rather than on $\cpk$,
  as its relation with the
  involution is more transparent there.
\end{rem}

\section{Exotic structures via rational blow-down}
\label{sec:RatBlowDown}

Exotic structures on $Z_1\# 4\cpkk$ can be also constructed using the rational
blow-down operation. To show such a construction, first we
need to examine specific exotic structures on $E(1)$ which are
constructed by the application of two parallel rational blow-downs.

To this end, consider an elliptic fibration on $E(1)$ with two
fishtail fibers $F_1, F_2$ and two sections $s_1, s_2$ (the latter
being embedded $(-1)$-spheres).  Blow up the singular points in the two
fishtail fibers and the intersections $F_1\cap s_2$ and $F_2\cap
s_1$. In the resulting $E(1)\# 4\cpkk$ we find two copies of the curve
configuration of a $(-5)$-sphere intersecting a $(-2)$ sphere.
\begin{figure}[htb]
\begin{center}
\setlength{\unitlength}{1mm}
\includegraphics[height=3.5cm]{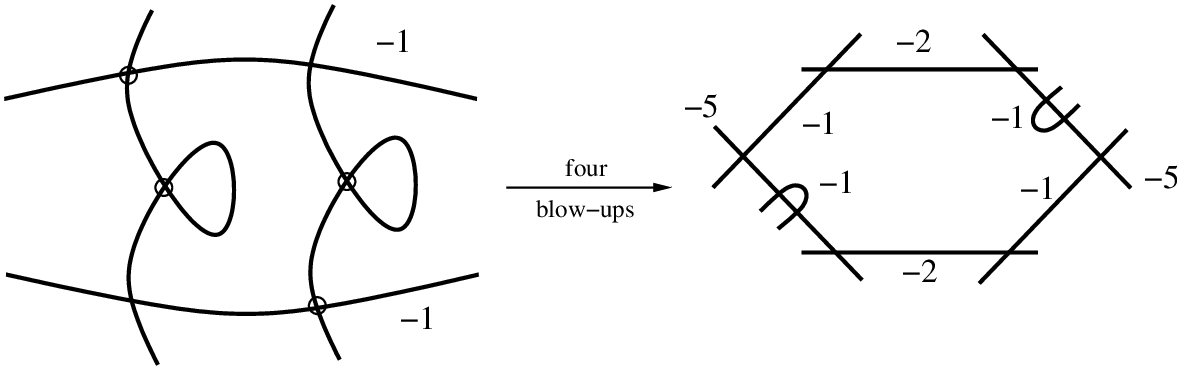}
\end{center}
\caption{\quad The two fishtail fibers are blown up twice each.
  (The points to be blown up are marked with  hollow circles on the left.)
  The exceptional
  divisors intersecting the $(-5)$-spheres twice are also shown, alhtough will
not be used in the constructions.}
\label{fig:fibration}
\end{figure}
These two configurations can be rationally blown down, as it is
described in \cite{FSRatBl}. If the starting fibration contains two
further singular fibers, one of which containing a $(-2)$-sphere
intersecting $s_1$ transversely once and disjoint from $s_2$, the
other one another $(-2)$-sphere intersecting $s_2$ transversely once
and disjoint from $s_1$, then the result of the blow-down will be
simply connected. Simple Euler characteristic and signature
calculation (together with Freedman's Classification
Theorem~\cite{Fr}) shows that the resulting four-manifold $D_{2}$ is
homeomorphic to $E(1)$. A simple Seiberg-Witten invariant calculation
shows
\begin{prop}\label{prop:exoticD2}
  $D_2$ is an exotic $E(1)$, i.e. these four-manifolds are not diffeomorphic
  to each other.
  \end{prop}
(The proof of this statement is a special case of the theorem we will prove
below.)

Consider now the elliptic fibration constructed earlier in
Section~\ref{sec:another}, which has two $I_4$ fibers and four
fishtail fibers.  Consider two fishtail fibers so that the involution
maps one into the other, and fix two sections which are also
interchanged by the involution.  Fix also two $(-2)$ curves, one from
each $I_4$ fibers which intersect one and the other sections and are
mapped into each other by the involution.  Blowing up as above four
times, the involution extends to this blow-up, and then to the
rationally blown-down manifold $D_2$. By taking the quotient, we get a
four-manifold which is homeomorphic to $Z_1\# 4\cpkk$, but (as the
universal covers are smoothly distinct) not diffeomorphic to it.

The same operation can be conveniently described 'downstairs' in $Z_1\#
4\cpkk$: the image of a fishtail fiber of $E(1)$ in $Z_1\# 4\cpkk$ gives
rise to an immersed sphere with one positive double point (and trivial in
homology), and a section in $E(1)$ defines a $(-1)$-sphere
intersecting the image of the fiber twice, with opposite
signs. Blowing up the double point of the immersed sphere and the
positive intersection we get the configuration of a $(-5)$-sphere and
a $(-2)$-sphere intersecting each other once.  Applying the rational
blow-down construction here (and using the image of an appropriate
$(-2)$-curve in the $I_4$ fiber upstairs to show that the rational
blow-down leaves the fundamental group intact) we get an exotic
structure on $Z_1\#4\cpkk$. This structure has the property that the
universal cover is an exotic symplectic manifold.  (Exotic structures
with this property can be also constructed by choosing a nontrivial
fibered knot in the proof of Theorem~\ref{thm:main}.)

In a similar manner, by applying further blow-ups, and blowing down
further configurations, we get an infinite family of exotic structures
on $Z_1\# 4\cpkk$.  Indeed, for such constructions we need to further
examine exotic structures on $E(1)$. In the above argument we found
six spheres in $E(1)\# 4\cpkk$ intersecting each other in a circular manner (as
shown by Figure~\ref{fig:fibration})
with self-intersections
$(-5),(-2),(-1),(-5),(-2),(-1)$; we deleted the two $(-1)$'s from the
configuration and blew down the remaining four curves. One can apply further
blow-ups and get longer chains to blow down.
If we first blow up the intersections of the $(-1)$
and $(-2)$ curves in the configuration of six (shown in
Figure~\ref{fig:blowups}),
\begin{figure}[htb]
\begin{center}
\setlength{\unitlength}{1mm}
\includegraphics[height=3.5cm]{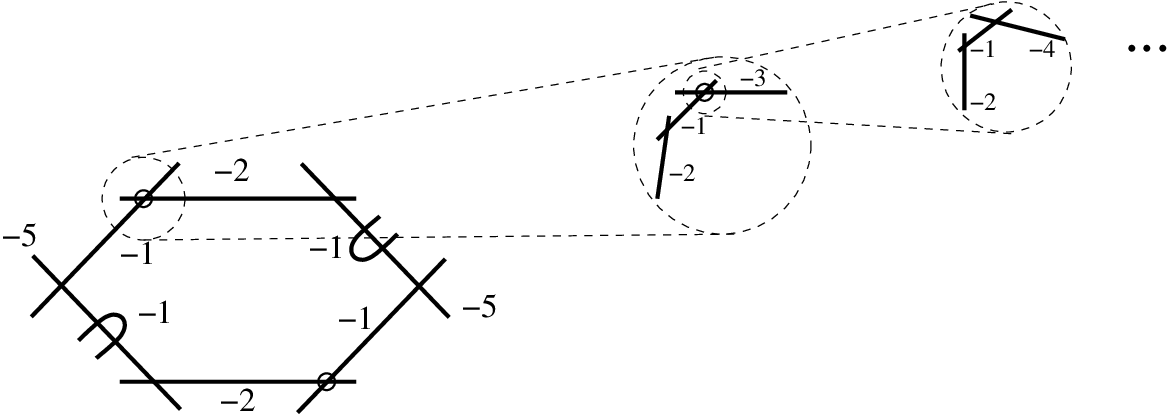}
\end{center}
\caption{\quad Repeated blow-ups of the $(-1)$- and $(-n+1)$-spheres
  are shown. The same sequence of blow-ups is applied in the intersection
of the other $(-1)$- and $(-2)$-sphere.}
\label{fig:blowups}
\end{figure}
and then blow up repeatedly (and in pairs) the intersection point of the new
$(-1)$-curve and the $(-n+1)$-curve, we get a configuration of two disjoint
sets of spheres ${\mathcal {C}}_1$ and
${\mathcal {C}}_2$ with intersection patterns $(-n), (-5), (-2), \ldots,
(-2)$ with $n-2$ $(-2)$-spheres.  We get these configurations after
blowing up $E(1)$ $2n$-times, hence in $\cpk \# (9+2n)\cpkk$.
Blowing down the two configurations ${\mathcal {C}}_1$ and ${\mathcal
  {C}}_2$ above, we get the four-manifold $D_n$. The usual
argument (using simple connectivity, guaranteed by the transverse
$(-2)$-sphere from the $I_4$ fiber, intersecting
the chain at one of its  ends) based on \cite{Fr} shows that
$D_n$ is homeomorphic to $E(1)$. Notice that for $n=2$ we recover our
earlier $D_2$.

\begin{prop}\label{prop:distinctDn}
  The four-manifolds $D_n$ for $n \geq 2$ are pairwise non-diffeomeomorphic
  symplectic four-manifolds.
\end{prop}

Before providing the proof, we point out the obvious corollary of this
result.  As the construction is done in a $\iota$-invariant manner, we
can take the quotient of $D_n$ by $\iota$; the result will be denoted
by $G_n$.

\begin{cor}
  There are infinitely many distinct smooth structures on
  $Z_1\# 4\cpkk$ which all have symplectic universal covers, and
  which can be constructed from $Z_1\# 4\cpkk$ by applying the
  rational blow down construction.
\end{cor}
\begin{proof}
  Indeed, the $G_n$ above have $D_n$ as universal cover, and those
  manifolds are symplectic. As the $D_n$'s are all smoothly distinct,
  so are the $G_n$'s. As before, the two rational blow-downs defining
  $D_n$ from $E(1)\# 2n\cpkk$ can be performed downstairs, after
  $Z_1\#4\cpkk$ is blown up $n$ times; this rational blow-down
  provides $G_n$, concluding the argument.
  \end{proof}

The interest in these manifolds stems from the fact that this
construction can be shown to be new when compared to our earlier proof
of Theorem~\ref{thm:main}.  For the sake of brevity, we will verify
this statement for $D_n$ only with even $n$.

\begin{thm}\label{thm:NotKnotSurg}
  There is no knot $K$ such that two knot surgeries on $E(1)$, both
  with $K$, results in a manifold diffeomorphic to $D_n$ for any even
  $n\geq 2$.
\end{thm}
The proof of the statements above will rely on the detailed
understanding of the Seiberg-Witten invariants of $D_n$.

\begin{proof}[Proof of Proposition~\ref{prop:distinctDn}]
  As the rational blow-down operation of symplectic spheres (which are
  symplectically orthogonal) is symplectic \cite{Sym},
  it follows that $D_n$ is symplectic for all $n$.

  In order to smoothly distinguish these four-manifolds, we will
  partially determine their Seiberg-Witten invariants.  Let $K$ denote
  the canonical class on $\cpk \# (9+2n)\cpkk$.  Restrict it to the
  complement of the two curve configurations ${\mathcal {C}}_1$ and
  ${\mathcal {C}}_2$, and extend over the rational disks (the
  extension exists, as the rational blow-down is a symplectic
  operation) resulting in a cohomology class in $D_n$, also denoted by
  the same letter $K$.  This class corresponds to the spin$^c$
  structure induced by the symplectic form. As the small perturbation
  Seiberg-Witten invariant is well-defined on manifolds with $b_2^+=1$
  and $b_2^-\leq 9$, we can talk about basic classes of $D_n$.  We
  will verify the following three facts:
  \begin{enumerate}
  \item The class $\pm K$ is a Seiberg-Witten basic class of $D_n$.
  \item The class $\pm K$ is $(2n-3)$-times a primitive class in
    $H^2(D_n; \Z )$.
  \item If a characteristic class $L$ is not a rational multiple of
    $K$, or if $L=r\cdot K$ with
    $\vert r\vert>1$, then $L$ is not a basic class of $D_n$.
  \end{enumerate}
  The above properties then show that (through Seiberg-Witten basic classes)
  the divisibility of $K\in H^2(D_n; \Z )$ (which is equal to 
  $(2n-3)$) is a smooth invariant of $D_n$, concluding the proof.
  
  In order to verify the above statements, we need some more notation.
  As customary, we will work with homologies rather than cohomologies
  (identified them via Poincar\'e duality),
  and will denote the generator of $H_2(\cpk ; \Z)$ by $h$ and the
  homology classes of the exceptional divisor of the $i^{th}$ blow-up
  by $e_i$ ($i=1, \ldots, 9+2n$). The homology class of the fiber $F$
  of the elliptic fibration $E(1)$ 
  in this basis is equal to
  \[
  F=3h-\sum _{i=1}^9 e_i.
  \]
  If we denote the two sections used in the construction by $e_1,
  e_2$, then the spheres we will blow down are given by
\[
e_2-\sum _{i=6}^{4+n}e_{2i},F-2e_{11}-e_{13}, e_{13}-e_{15}, e_{15}-e_{17},
\ldots , e_{7+2n}-e_{9+2n}
\]
and 
\[
e_1-\sum _{i=6}^{4+n}e_{2i+1}, F-2e_{10}-e_{12}, e_{12}-e_{14}, e_{14}-e_{16},
\ldots , e_{6+2n}-e_{8+2n}.
\]
(Here the dots mean sequences of $e_i$'s with $i$ of the same
parity throughout the sequence.)
To verify (1), we need to show that
in $\cpk \# (9+2n)\cpkk$ there is a wall separating $h$ (which has its
Poincar\'e dual near the period point of the positive scalar curvature
metric, hence the invariant is 0), and
\[
K=3h-\sum _{i=1}^{9+2n}e_i \in H_2(\cpk \# (9+2n)\cpkk; \Z).
\]
To achieve this, we need a homology class, which is orthogonal
to all spheres in the configuration to be blown down, has positive
square, and pairs with $h$ and $K$ with opposite signs.
(Note that $h\cdot K=3>0$.) 
It is easy to
see that the class $x=(3n-2)h-(n-1)\sum _{i=1}^9e_i-\sum
_{i=10}^{9+2n}e_i$
satisfies all these conditions: $x^2=4n-5>0$, $x\cdot K=3-2n<0$ and $x\cdot h=
3n-2>0$, as $n\geq 2$ has been assumed.

To show that $K$ in $H_2(D_n; \Z )$ is divisible by $(2n-3)$, add
\begin{equation}\label{eq:ToAdd}
(n-2)(f-2e_{11}-e_{13})+(n-2)(f-2e_{10}-e_{12})+(n-2)(e_{15}-e_{13})+
  (n-2)(e_{14}-e_{12})+
\end{equation}
\[
+ (n-3)(e_{17}-e_{15}) +(n-3)(e_{16}-e_{14})+(n-4)(e_{19}-e_{17})
+(n-4)(e_{18}-e_{16})+\ldots +
\]
\[
+(e_{9+2n}-e_{7+2n})+(e_{8+2n}-e_{6+2n})
\]
to $K$ in $\cpk \# (9+2n)\cpkk$, and notice that the resulting class
is obviously divisible by $(2n-3)$. As the classes in
Equation~\eqref{eq:ToAdd} can be represented in the plumbing part (to
be blown down), it follows that the part of $K$ in the complement of
the plumbings is divisible. Simple homological argument shows that the
extension over the rational homology balls will be divisible by the
same quantity.  This shows that the divisibility of $K$ is a multiple
of $(2n-3)$, but since $x\cdot K=3-2n$, the divisiblility is equal to
$(2n-3)$.

Finally, suppose that $L$ is
another basic class. As by \cite{TaubesGromov}
the Poincar\'e dual of $-K$ can be represented by
a torus of self-intersection 0, 
the adjunction formula gives that $K\cdot L=0$.
As $D_n$ has $b_2^+=1$, this implies that $L=rK$ for some
rational number $r$.
In this latter case \cite[Theorem~2]{Taubes} shows that
$\vert r \vert \leq 1$, concluding the proof.
\end{proof}


\begin{proof}[Proof of Theorem~\ref{thm:NotKnotSurg}]
  that $D_{2k}$ can be given as two knot surgeries on $E(1)$ with knot
  $K$, having symmetrized Alexander polynomial $\Delta _K$ of degree
  $d$. Then by \cite{FSknot} the extremal basic class (i.e. the basic
  class which is the highest multiple of a primitive class) of
  $E(1)_{K,K}$ is $(4d-1)$-times a primitive class. On the other hand,
  the proof of Proposition~\ref{prop:distinctDn} showed that for
  $D_{2k}$ the extremal basic class is $(4k-3)$-times a primitive
  class. As these numbers differ mod 4, we get the desired
  contradiction.
\end{proof}

\begin{rem}
  For $D_n$ with $n$ odd the same statemnet holds, but the argument is
  slighly more involved, and requires a thorough understanding of
  further basic classes of the manifold.
\end{rem}

The involution of Section~\ref{sec:another} induces a similar
involution on the elliptic surfaces $E(2n+1)$: consider the definition
\[
E(2n+1)=E(n)\#_f E(1)\# _f E(n),
\]
where $\# _f$ denotes the fiber sum operation. We apply this operation
along two fibers of $E(1)$ which are mapped into each other via the
involution $\iota$, and extend this map to the $E(n)$'s (also denoted by
$\iota$).

Then applying the same strategy (pair of knot surgeries along
$\iota$-equivalent fibers and then taking the quotient), we get exotic
smooth structures on $Z_1\#2n\cpk \# (10n+4)\cpkk$.  Note that these
manifolds have signature $-8n-4$, so are never spin (and their
universal covers with $\sigma =-8(2n-1)$ are also non-spin).  Note
also that their $b_2^+$-invariant is always even. A simple argument
(resting on the double node surgery method of Fintushel and Stern and
rational blow down) leads to exotic four-manifolds with the same
(even) $b_2^+$, but with slighly smaller Euler characteristic.

\begin{thm}
  There are infinitely many smooth structures on 
  $Z_1\# 2n\cpk \# k\cpkk$ with $n\in \N$ and $k\geq 8n$.
\end{thm}
\begin{proof}
  We will use the fact that our original fibration on $E(1)$ has two
  $I_4$ fibers. Indeed, consider two sections $s_1, s_2$ as before
  (mapped by $\iota$ into each other), and apply the knot surgery
  operation along two smooth fibers (interchanged by $\iota$) using
  the twist knots $K_m$ in both cases.  In this construction, the
  sections will be glued to a Seifert surface of the knot; as
  $g_3(K_m)=1$, the result is two genus-2 surfaces originated from
  $s_1$ and $s_2$.  As it is explained in \cite{FSDoubleNode}, if
  there is an $I_2$ fiber near the knot surgery, one genus on one of
  the curves can be modified to a positive double point. With an $I_4$
  fiber, the same can be done along both sections. With the two $I_4$
  fibers we can perform this operation along both knot surgeries in
  such a manner that the resulting spheres with two positive double
  points will be mapped by $\iota$ into each other.  This shows that
  in $E(2n+1)$ we have two 2-spheres with square $-2n-1$ and with two
  positive double points. Blowing up the four double points we get two
  2-spheres (again, interchanged by $\iota$) of self-intersection
  $-2n-9$.  Consider a fibration on $E(n)$ before the fiber sum with
  an $I_{2n+7}$ fiber (such fibration obviously exists). In $E(2n+1)$
  therefore we find two copies of the sphere configurations $(-2n-9),
  (-2), \ldots , (-2)$ (with $2n+5$ $(-2)$'s), which can be blown
  down. As there is a futher sphere intersecting the last $(-2)$ in
  the above chain, the result is simply connected. A simple
  application of the calculation determining the effect of the
  rational blow down process on Seiberg-Witten invariants implies that
  the results are all smoothly distinct (detected by the leading
  coefficients of the Alexander polynomials of the chosen twist
  knots).  Taking the quotient by $\iota$, we get exotic structures on
  $Z_1\# 2n\cpk \# 8n\cpkk$.  Choosing different twist knots, the
  usual argument shows that the resulting smooth structures will be
  different.  Finally, we can always blow up the resulting exotic
  manifolds, and by the blow-up formula for the universal covers it
  follows that the resulting manifolds (before and after taking
  the quotient with the $\iota$-action) will be exotic.
\end{proof}

\end{document}